\documentclass[12pt]{amsart}
\usepackage{amscd,amsmath,amsthm,amssymb,graphics}
\usepackage{pstcol,pst-plot,pst-3d}
\usepackage{lmodern,pst-node}
\usepackage[dvips]{graphicx}
\usepackage{multicol}
\usepackage{epic,eepic}
\usepackage{amsfonts,amssymb,amscd,amsmath,enumitem,verbatim}
\psset{unit=0.7cm,linewidth=0.8pt,arrowsize=2.5pt 4}

\newpsstyle{fatline}{linewidth=1.5pt}
\newpsstyle{fyp}{fillstyle=solid,fillcolor=verylight}
\definecolor{verylight}{gray}{0.97}
\definecolor{light}{gray}{0.9}
\definecolor{medium}{gray}{0.85}
\definecolor{dark}{gray}{0.6}

\usepackage{lineno}

\def\frk{\frak}               

\def\mm{{\frk m}}

\def\Phi{{\frk n}}
\def\Phi{{\frk N}}
\def\MR{{\mathcal R}}

\def\MS{{\mathcal S}}

\def\MC{{\mathcal C}}

\def\ab{{\bold a}}
\def\bb{{\bold b}}

\def\gb{{\bold g}}

\def\tb{{\bold t}}
\def\opn#1#2{\def#1{\operatorname{#2}}} 
\opn\chara{char} \opn\length{\ell} \opn\pd{pd} \opn\rk{rk}
\opn\projdim{proj\,dim} \opn\injdim{inj\,dim} \opn\rank{rank}
\opn\depth{depth} \opn\grade{grade} \opn\height{height}
\opn\embdim{emb\,dim} \opn\codim{codim}

\opn\Tr{Tr} \opn\bigrank{big\,rank}
\opn\superheight{superheight}\opn\lcm{lcm}
\opn\trdeg{tr\,deg}
\opn\reg{reg} \opn\lreg{lreg} \opn\ini{in} \opn\lpd{lpd}
\opn\size{size}\opn\bigsize{bigsize}
\opn\cosize{cosize}\opn\bigcosize{bigcosize}
\opn\sdepth{sdepth}\opn\sreg{sreg}
\opn\link{link}\opn\fdepth{fdepth}
\opn\div{div} \opn\Div{Div} \opn\cl{cl} \opn\Cl{Cl}
\let\epsilon\varepsilon
\let\phi=\varphi
\let\kappa=\varkappa
\opn\Spec{Spec} \opn\Supp{Supp} \opn\supp{supp} \opn\Sing{Sing}
\opn\Ass{Ass} \opn\Min{Min}\opn\Mon{Mon} \opn\dstab{dstab} \opn\astab{astab}
\opn\Syz{Syz}
\opn\Ann{Ann} \opn\Rad{Rad} \opn\Soc{Soc}
\opn\Im{Im} \opn\Ker{Ker} \opn\Coker{Coker} \opn\Am{Am}
\opn\Hom{Hom} \opn\Tor{Tor} \opn\Ext{Ext} \opn\End{End}
\opn\Aut{Aut} \opn\id{id}

\opn\nat{nat}
\opn\pff{pf}
\opn\Pf{Pf} \opn\GL{GL} \opn\SL{SL} \opn\mod{mod} \opn\ord{ord}
\opn\Gin{Gin} \opn\Hilb{Hilb}\opn\sort{sort}
\opn\initial{init}
\opn\ende{end}
\opn\height{height}
\opn\type{type}
\opn\set{set}
\opn\indeg{indeg}
\opn\aff{aff} \opn\con{conv} \opn\relint{relint} \opn\st{st}
\opn\lk{lk} \opn\cn{cn} \opn\core{core} \opn\vol{vol}
\opn\link{link} \opn\star{star}\opn\lex{lex}\opn\cochord{co-chord}\opn\im{im}
\opn\gr{gr}

\def\pot#1#2{#1[\kern-0.28ex[#2]\kern-0.28ex]}

\opn\dirlim{\underrightarrow{\lim}}
\opn\inivlim{\underleftarrow{\lim}}
\let\to=\rightarrow

\def\Implies{\ifmmode\Longrightarrow \else
        \unskip${}\Longrightarrow{}$\ignorespaces\fi}
\def\implies{\ifmmode\Rightarrow \else
        \unskip${}\Rightarrow{}$\ignorespaces\fi}
\def\iff{\ifmmode\Longleftrightarrow \else
        \unskip${}\Longleftrightarrow{}$\ignorespaces\fi}

\let\:=\colon
 \theoremstyle{plain}
\newtheorem{Theorem}{Theorem}[section]
 \newtheorem{Lemma}[Theorem]{Lemma}
 \newtheorem{Corollary}[Theorem]{Corollary}
 \newtheorem{Proposition}[Theorem]{Proposition}
 \newtheorem{Problem}[Theorem]{Problem}
 \newtheorem{Conjecture}[Theorem]{Conjecture}
 \newtheorem{Question}[Theorem]{Question}
 \theoremstyle{definition}
 
 \newtheorem{Remark}[Theorem]{Remark}

\let\epsilon\varepsilon
\let\kappa=\varkappa
\opn\dis{dis}
\def\pnt{{\raise0.5mm\hbox{\large\bf.}}}

\opn\Lex{Lex}

\begin{document}
\title{Powers of  binomial edge ideals with quadratic Gr\"obner bases}
\author {Viviana Ene, Giancarlo Rinaldo, Naoki Terai}

\address{Viviana Ene, Faculty of Mathematics and Computer Science, Ovidius University, Bd. Mamaia 124, 900527 Constanta, Romania}  \email{vivian@univ-ovidius.ro}

\address{Giancarlo Rinaldo, Department of Mathematics,  University of Trento, Via Sommarive,
14  38123 Povo (Trento), Italy}  \email{giancarlo.rinaldo@unitn.it}

\address{Naoki Terai, Department of Mathematics, Okayama University, 3-1-1, Tsushima-naka, Kita-ku, Okayama, 700-8530, Japan}  \email{terai@okayama-u.ac.jp}

\thanks{The third author was supported by the JSPS Grant-in Aid for Scientific Research (C)  18K03244. }

\begin{abstract}
We study powers of binomial edge ideals associated with closed and block graphs. 
\end{abstract}

\subjclass[2010]{13D02, 05E40, 14M05}
\keywords{Binomial edge ideals, chordal graphs, depth function, regularity, symbolic powers}

\maketitle
\section*{Introduction}
\label{S:Intro}

Binomial edge ideals generalize in a natural way the determinantal ideals gene\-rated by the $2$--minors of a generic matrix of type 
$2\times n.$ They were independently introduced a decade ago in the papers \cite{HHHKR} and \cite{Ohtani}. Since then, they have been intensively studied and there exists a rich recent literature on this subject. Fundamental results on their Gr\"obner bases, primary decomposition, and their resolutions are presented in the monograph \cite{HHO}.

Binomial edge ideals with quadratic Gr\"obner basis are of particular interest since their initial ideals are monomial edge ideals associated with bipartite graphs. Therefore, the theory of monomial edge ideals can be employed in deriving information about binomial edge ideals. 

While many questions regarding binomial edge ideals have been already answered, much less is known about their powers. In \cite{JKS}, 
first steps in studying the  regularity of powers of binomial edge ideals have been done. By using quadratic sequences, the authors obtain bounds for the regularity of powers of binomial edge ideals which are almost complete intersection. For the same class of ideals,  in the paper
\cite{JKS2}, the Rees rings are considered.  Another direction of research was an  approach in \cite{EH}. Here it is shown that binomial edge ideals with quadratic Gr\"obner basis have the nice property that their ordinary and symbolic powers coincide.  
\medskip

Let $G$ be a simple graph (i.e. an undirected graph with no multiple edges and no loops) on the vertex set $[n]=\{1,2,\ldots,n\}$ and 
let $S=K[x_1,x_2,\ldots,x_n,y_1,y_2,\ldots,y_n]$ be the polynomial ring in $2n$ variables over the field $K.$ For 
$1\leq i<j\leq n,$ we set $f_{ij}=x_iy_j-x_jy_i.$ The binomial edge ideal of the graph $G$ is 
\[J_G=(f_{ij}: i<j, \{i,j\} \text{ is an edge in }G).\] 

We consider the polynomial ring $S$ endowed with the lexicographic order induced by the natural order of the variables, namely 
$x_1>x_2>\cdots>x_n>y_1>y_2\cdots >y_n.$ The Gr\"obner basis of $J_G$ with respect to this order was computed in \cite{HHHKR}. 
The graphs $G$ with the property that $J_G$ has a quadratic Gr\"obner basis were characterized in the same paper and they were called closed. Later on, it turned out that closed graphs coincide with the proper interval graphs which have a history of about 50 years in combinatorics. In Section~\ref{S:Prelim}, we survey various combinatorial characterizations of closed graphs which are very useful in working with their associated binomial edge ideals. Closed graphs with Cohen-Macaulay binomial edge ideals are classified in \cite{EH}.
Roughly speaking, they are "chains" of cliques (i.e. complete graphs) with the property that every two consecutive cliques intersect in one vertex. For Cohen-Macaulay binomial edge ideals of closed graphs we compute the depth function in 
Theorem~\ref{thm:powersclosed} and Proposition~\ref{pr:discdepth}. For this class of ideals, we show that 
\[
\depth\frac{S}{J_G^i}=\depth\frac{S}{\ini_<(J_G^i)}
\] and this common value depends on the cardinality of the maximal cliques of $G.$ The proof of Theorem~\ref{thm:powersclosed} follows from several technical lemmas. The basic idea of the proof is the following. Starting with a closed graph $G$ whose binomial edge ideal
is Cohen-Macaulay, we consider a disconnected graph $G'$ whose connected components are complete graphs of the same size as the maximal cliques of $G.$ By using the techniques developed in \cite{HTT} for computing the depth of powers of sums of ideals, we are able to calculate the depth of the powers of $J_{G'}.$ Next, by using a regular sequence of linear forms, we can recover the powers of 
$J_G$ from the powers of $J_{G'}$, and, finally we can compute the depth of the powers of $J_G.$ Similar arguments are used to compute the depth for the powers of $\ini_<(J_G).$

In addition, Proposition~\ref{pr:discdepth} implies that the depth function of $J_G$ and $\ini_<(J_G)$ is non-increasing. We expect the same behavior for every closed graph, not only for those whose binomial edge ideal is Cohen-Macaulay; see Question~\ref{Q1}. However, we are able to show that for every closed graph $G$, the depth limit for $J_G$ and 
$\ini_<(J_G)$ coincide and we compute this value in Theorem~\ref{th:limdepth}. On the other hand, in Proposition~\ref{pr:nondecr} we show that the initial ideal $\ini_<(J_G)$ has a non-increasing depth function. An important step  in deriving 
Theorem~\ref{th:limdepth} is
Proposition~\ref{pr:Rees} where we prove that the Rees rings $\MR(J_G)$ and $\MR(\ini_<(J_G))$ are Cohen-Macaulay. This reduces the proof of the equality 
\[
\lim_{k\to\infty}\depth\frac{S}{J_G^k}=\lim_{k\to\infty}\depth\frac{S}{(\ini_<(J_G))^k}
\] by showing that $J_G$ and $\ini_<(J_G)$ have the same analytic spread. This is shown by using the Sagbi basis theory.

One of the problems that we have considered at the beginning of this project was to characterize the graphs $G$ such that 
$J_G ^k$ is Cohen-Macaulay  for (some) $k\geq 2.$ We still do not have a complete solution for this problem which is probably very difficult in the largest generality, but we can solve it if we restrict to closed or connected block graphs; see  
Proposition~\ref{pr:closednotCM} and Proposition~\ref{pr:nonCMblock}. In the last part of Section~\ref{S:Depth} we show that 
binomial edge ideals have the strong persistence property, as they initial ideals do.   
\medskip

In Section~\ref{S:Reg}, we compute the regularity function for $J_G$, when $G$ is closed. This is done in Theorem~\ref{thn:regclosed} where we prove that if $G$ is connected, then, for every $k\geq 1,$ 
\[
\reg\frac{S}{J_G^k}=\reg\frac{S}{\ini_<(J_G^k)}=\ell+2(k-1),
\] where $\ell$ is the length of the longest induced path in $G.$ The inequality $\reg S/J_G^k\geq \ell+2(k-1)$ follows from a result in 
\cite{JKS}. For the rest of the proof, we combine various known facts about the regularity of the powers of edge ideals of bipartite graphs. The statement is extended to disconnected closed graphs in  Proposition~\ref{pr:regdisconnect}. 

In Section~\ref{S:Block}, we consider block graphs. These are chordal graphs with the property that every two maximal cliques intersect in at most one vertex. For the block graphs whose binomial edge ideal is Cohen-Macaulay we show that the symbolic powers coincide with the ordinary ones if and only if the graph is closed. This theorem shows, in particular, that the equality between the symbolic and ordinary powers of binomial edge ideals does not hold for all chordal graphs. Finally, in Proposition~\ref{pr:nonCMblock}, we show that for every connected block graph $G$ which is not a path, $J_G^k$ is not Cohen-Mcaulay for $k\geq 2.$

In the last section of the paper we discuss a few open questions. The most intriguing is related to a conjecture which appeared 
in \cite{EHH} and which is still open. This conjecture states that for every closed graph $G,$ we have 
$\beta_{ij}(J_G)=\beta_{ij}(\ini_<(J_G)).$  While doing some calculations with the computer, we observed an interesting phenomenon, namely that the graded Betti numbers are the same also for  powers of $J_G$ and $\ini_<(J_G).$ Moreover, as we explain in Section~\ref{S:Open}, the equalities between the graded Betti numbers  are true for complete and path graphs. Taking into account also our results on the regularity and depth of the powers of $J_G$ and $\ini_<(J_G),$ we 
were tempted to conjecture that for every closed graph $G$ and every $k\geq 1,$ we have $\beta_{ij}(J_G^k)=\beta_{ij}((\ini_<(J_G))^k).$ 

Another interesting question concerns the block graphs. By computer calculation, we observed that the net  graph $N$ 
(Figure~\ref{H1andH2}) which plays an important role in Theorem~\ref{thm:netfree} has the  property that $J_N^{(2)}$ is Cohen-Macaulay. On the other hand, $J_N^2$ is not Cohen-Macaulay. It would be of interest to classify all the block graphs with the property that the second symbolic power of the associated binomial edge ideal is Cohen-Macaulay.

\section{Preliminaries}
\label{S:Prelim}

Let $G$ be a graph \footnote {In this paper, by a graph we always mean a simple graph, that is, an undirected graph with no multiple edges and no loops.} on the vertex set $V(G)=[n]$ and edge set $E(G).$    Let $S=K[x_1,\ldots,x_n,y_1,\ldots,y_n]$ be the polynomial 
ring in $2n$ variables over the field $K.$ The binomial edge ideal $J_G$ associated with $G$ is generated by the binomials 
$f_{ij}=x_iy_j-x_jy_i\in S$ where $\{i,j\}\in E(G).$ In other words, $J_G$ is generated by the maximal minors of the generic 
$2\times n$-matrix $X=\left(
\begin{array}{cccc}
	x_1 & x_2 & \cdots & x_n\\
	y_1 & y_2 & \cdots & y_n
\end{array}\right)
$ determined by the edges of $G.$ The most common examples of binomial edge ideals are the ideal $I_2(X)$ generated by all the 
maximal minors of $X$ which is equal to $J_{K_n}$ where $K_n$ is the complete graph on the vertex set $[n],$ and the ideal generated by the adjacent minors of $X$ which coincides with $J_{P_n},$ where $P_n$ is the path graph with edges $\{i,i+1\}, 1\leq i \leq n-1.$

We consider the polynomial ring $S$ endowed with the lexicographic order induced by the natural order of the variables. Let $\ini_<(J_G)$ be the initial ideal of $J_G$ with respect to this monomial order. By \cite[Corollary 2.2]{HHHKR}, $J_G$ is a radical ideal. In the same paper, it was shown that the  minimal prime ideals can be characterized in terms of the combinatorics of the graph $G.$ In order to recall this characterization, we introduce the following notation. 
Let $W\subset [n]$ be a (possible empty) subset of $[n]$, and let $G_1,\ldots,G_{c(W)}$ be the connected components of $G_{[n]\setminus W}$ where 
$G_{[n]\setminus W}$ is the induced subgraph of $G$ on the vertex set $[n]\setminus W.$ For $1\leq i\leq c(W),$ let $\tilde{G}_i$ be the complete 
graph on the vertex set $V(G_i).$ Let \[P_{W}(G)=(\{x_i,y_i\}_{i\in W}) +J_{\tilde{G_1}}+\cdots +J_{\tilde{G}_{c(\MS)}}.\]
Then $P_{W}(G)$ is a prime ideal of height equal to $n-c(W)+|W|,$ for every $W\subset [n].$

By \cite[Theorem 3.2]{HHHKR}, $J_G=\bigcap_{W\subset [n]}P_{W}(G).$ In particular, the minimal primes of $J_G$ are among the prime ideals $P_{W}(G)$ with $W\subset [n].$

\begin{Proposition}\label{cpset}\cite[Corollary 3.9]{HHHKR}
$P_{W}(G)$ is a minimal prime of $J_G$ if and only if either $W=\emptyset$ or $W$ is non-empty and for each 
$i\in W,$ $c(W\setminus\{i\})<c(W)$.
\end{Proposition}

In graph theoretical terms, $P_{W}(G)$ is a minimal prime ideal of $J_G$ if and only if $W$ is empty or $W$ is non-empty and  is a 
\emph{cut-point set} of $G,$ that is, $i$ is a cut point of the restriction $G_{([n]\setminus W)\cup\{i\}}$ for every $i\in W.$ Let 
$\MC(G)$ be the set of all sets $W\subset [n]$ such that $P_{W}(G)\in \Min(J_G),$ where $\Min(J_G)$ is the set of minimal prime ideals of 
$J_G.$

In particular, it follows
\begin{equation}\label{eq:dimension}
\dim S/J_G=\max\{n+c(W)-|W|: W \in \MC(G)\}.
\end{equation}

For $W=\emptyset,$ $c=c(\emptyset)$ is the number of connected components of $G.$ In addition,  one easily sees that $P_{\emptyset}(G)$ is  a minimal prime of $J_G.$ Therefore, if $J_G$ is 
unmixed (which is the case, for instance, if $J_G$ is Cohen-Macaulay), then all the minimal primes of $J_G$ have   dimension equal to 
 $n+c.$ In particular, if 
$G$ is connected, then $J_G$ is unmixed if and only if, for every minimal prime $P_{W}(G)$ of $G,$ we have $n+c(W)-|W|=n+1,$ that is, 
$c(W)-|W|=1.$

By  \cite[Theorem 3.1]{CDeG} and \cite[Corollary 2.12]{CDeG}, we have 
\begin{equation}\label{eq:intersectini}
\ini_<(J_G)=\bigcap_{W \in \MC(G)} \ini_< P_{W}(G).
\end{equation}

In what follows, we are mainly interested in binomial edge ideals with quadratic Gr\"obner bases. We recall the following result from 
\cite{HHHKR}.

\begin{Theorem}\cite{HHHKR}
\label{th:quadraticGB}
Let $G$ be a  graph on the vertex set $[n]$ with the edge set $E(G)$, and let $<$ be the lexicographic order on $S$ induced by $x_1>\cdots > x_n> y_1>\cdots >y_n$.
Then the following conditions are equivalent:
\begin{enumerate}
	\item [{\em (a)}] The generators $f_{ij}$ of $J_G$ form a quadratic Gr\"obner basis.
	\item [{\em (b)}] For all edges $\{i,j\}$ and $\{i,k\}$ with $j>i<k$ or $j<i>k$ one has $\{j,k\}\in E(G)$.
\end{enumerate}
\end{Theorem}

According to \cite{HHHKR}, a graph $G$ endowed with a labeling which satisfies condition (b) in the above theorem is called {\em closed with respect to the given labeling}.
Therefore, the generators of $J_G $ form a Gr\"obner basis with respect to the lexicographic order if and only if $G$ is closed with respect to its given labeling. Moreover, a graph $G$ is called {\em closed} if there exists a labeling of its vertices such that $G$ is closed with respect to it. Later on, it turned out that closed graphs have a rich history in combinatorics and they are known as \textit{proper interval graphs}. However, in this paper we will use the terminology of the paper \cite{HHHKR}. There are several characterizations of closed graphs. 
 Before discussing them, let us recall some notions of graph theory. A graph is called \emph{chordal} if it has no induced cycle of length greater than or equal to $4.$ A graph is called \emph{claw-free} if it has no induced subgraph isomorphic with that one displayed in Figure~\ref{Fig:claw}.
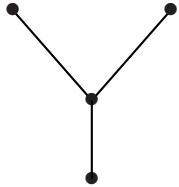
\begin{figure}[hbt]
\begin{center}
\psset{unit=1.5cm}
\begin{pspicture}(1.2,0.5)(7,2.5)
\psline(4.1,1.2)(3.4,2) 
\psline(4.1,1.2)(4.8,2) \rput(4.1,1.2){$\bullet$} \rput(3.4,2){$\bullet$} \rput(4.8,2){$\bullet$} \rput(4.1,0.5){$\bullet$}
\psline(4.1,1.2)(4.1,0.5) 
\end{pspicture}
\end{center}
\caption{Claw graph}\label{Fig:claw}
\end{figure}
A \emph{clique} of the graph $G$ is a complete subgraph of $G.$ The cliques of $G$ form a simplicial complex $\Delta(G)$ which is called the \emph{clique complex} of $G.$

The equivalences of the following theorem collects several results proved in \cite{BEI, CR, CR2, EHH, HHHKR}.

\begin{Theorem}\label{th:closedchar}
Let $G$ be a graph on the vertex set $[n].$ The following statements are equivalent:
\begin{itemize}
\item[\emph{(i)}] $G$ is a closed graph with respect to the given labeling, or equivalently,  the generators of $J_G$ form a Gr\"obner basis with respect to the lexicographic order induced by $x_1>\cdots >x_n>y_1>\cdots>y_n;$
\item [\emph{(ii)}] for all $\{i,j\},\{k,\ell\}\in E(G)$ with $i<j$ and $k<\ell,$ one has $\{j,\ell\}\in E(G)$ if $i=k, j\neq \ell$, and  
$\{i,k\}\in E(G)$ if $j=\ell, i\neq k;$
\item [\emph{(iii)}] The facets, say $F_1,\ldots, F_r,$ of the clique complex $\Delta(G)$ of $G$ are intervals of the form $F_i=[a_i,b_i]$ which can be ordered such that 
$1=a_1<\cdots < a_r<b_r=n;$
\item [\emph{(iv)}] for any $1\leq i<j<k\leq n,$ if $\{i,k\}\in E(G),$ then $\{i,j\},\{j,k\}\in E(G).$
\item [\emph{(v)}] $G$ is a chordal and claw-free graph which does not contain any subgraph isomorphic to the graphs displayed in 
Figure~\ref{H1andH2}.
\end{itemize}
\end{Theorem}

\begin{figure}[hbt]
\begin{center}
\psset{unit=0.8cm}
\begin{pspicture}(1,1)(5,7)
\rput(-3,1){
\pspolygon(2,2)(3,3.71)(4,2)
\psline(3,3.71)(3,5.2)
\psline(0.6,1.1)(2,2)
\psline(4,2)(5.6,1.1)
\rput(2,2){$\bullet$}
\rput(3,3.71){$\bullet$}
\rput(4,2){$\bullet$}
\rput(3,5.2){$\bullet$}
\rput(0.6,1.1){$\bullet$}
\rput(5.6,1.1){$\bullet$}
\rput(3.3,0.4){$H_1:$ Net graph}
}
\rput(4,1.5){
\psset{unit=1.5cm}
\pspolygon(0.5,0.5)(1.5,0.5)(1,1.4)
\pspolygon(1.5,0.5)(2,1.4)(2.5,0.5)
\pspolygon(1,1.4)(1.5,2.3)(2,1.4)

\rput(0.5,0.5){$\bullet$}
\rput(1.5,0.5){$\bullet$}
\rput(1,1.4){$\bullet$}
\rput(2,1.4){$\bullet$}
\rput(2.5,0.5){$\bullet$}
\rput(1.5,2.3){$\bullet$}
\rput(1.5,-0.06){ $H_2:$ Tent graph}
}
\end{pspicture}
\end{center}
\caption{}
\label{H1andH2}
\end{figure}

The connected closed graphs with Cohen-Macaulay binomial edge ideals were characterized in \cite[Theorem 3.1]{EHH}.

\begin{Theorem}\cite{EHH}
\label{Th:classification}
Let $G$ be a connected graph on $[n]$ which is closed with respect to the given labeling. Then the following conditions are equivalent:
\begin{itemize}
\item [(a)] $J_G$ is unmixed;
\item [(b)] $J_G$ is Cohen-Macaulay;
\item [(c)] $\ini_< (J_G)$ is Cohen-Macaulay;
\item [(d)] $G$ satisfies the following condition:  if  $\{i,j+1\}, \{j,k+1\}\in E(G)$ with $i<j<k$, then
	$\{i,k+1\}\in E(G)$;
\item [(e)] there exist integers $1=a_1<a_2<\cdots <a_r<a_{r+1}=n$ and a leaf order of the facets $F_1,\ldots,F_r$ of $\Delta(G)$ such that $F_i=[a_i,a_{i+1}]$ for all $i=1,\ldots,r$.
\end{itemize}
\end{Theorem}

Let us remark that if $G$ is closed and has the connected components $G_1,G_2,\ldots,,G_c,$ then 
\[
\frac{S}{J_G}\cong \frac{S}{J_{G_1}}\otimes \frac{S}{J_{G_2}}\otimes \cdots \otimes \frac{S}{J_{G_c}},
\] where $S_i=K[{x_j,y_j: j\in V(G_i)}]$ for $1\leq i\leq c.$ Thus $J_G$ is Cohen-Macaulay if and only if each $S_i/J_{G_i}$  is 
Cohen-Macaulay. 

Let $G$ be a closed graph. Then the generators of $J_G$ form the reduced Gr\"obner basis with respect to the lexicographic order. 
This implies that $\ini_<(J_G)=(x_iy_j: i<j,  \{i,j\}\in E(G)),$ thus $\ini_<(J_G)$ is the monomial edge ideal of a bipartite graph, let us call it $H,$
on the vertex set $\{x_1,x_2,\ldots,x_n\}\cup \{y_1,y_2,\ldots,y_n\}$ whose edges are $\{x_i,y_j\}$ where $\{i,j\}\in E(G).$
Since $H$ is bipartite, it follows that the edge ideal $I(H)=\ini_<(J_G)$ has the property that its ordinary powers coincide with the symbolic ones. 
Combining (\ref{eq:intersectini}) with   the proof of \cite[Lemma 2.1]{EH}, it follows that if $G$ is closed, then 
\begin{equation}\label{eq:lemmaEH}
\ini_<(J_G^i)=(\ini_<J_G)^i, \text{ for every }i\geq 1.
\end{equation}
In other words, if $G$ is closed, then the generators of $J_G^i$ form a Gr\"obner basis of $J_G^i$ for $i\geq 1.$ Moreover, in the same hypothesis on the graph $G$, by 
\cite[Corollary 2.4, Proposition 1.5]{EH}, we  have
\begin{equation}\label{eq:eqpowers}
J_G^i=J_G^{(i)}\text{ for every }i\geq 1.
\end{equation}

\section{Depth of powers}
\label{S:Depth}

The first main result of this section is the following.

\begin{Theorem}\label{thm:powersclosed}
Let $G$ be a connected closed graph on the vertex set $[n]$ such that $J_G$ is Cohen-Macaulay. Let $F_1,F_2,\ldots,F_r$ be the maximal cliques of 
$G$ and $d_i=\dim F_i=\# F_i-1$ for $1\leq i\leq r.$ Assume that $d_1\geq d_2\geq \cdots\geq d_r\geq 1.$ \footnote{Note that this is not necessarily the order of the facets of $\Delta(G)$ from Theorem~\ref{th:closedchar}} Then the following equalities hold: 
\begin{itemize}
	\item [\emph{(a)}] \[\depth\frac{S}{J_G^i}=\depth\frac{S}{\ini_<(J_G^i)}=n-\sum_{j=1}^{i-1}d_j+i, \text{ for }1\leq i\leq r,\]
	\item [\emph{(b)}] \[\depth\frac{S}{J_G^i}=\depth\frac{S}{\ini_<(J_G^i)}=r+2, \text{ for } i\geq r+1.\]
\end{itemize}
\end{Theorem}

For the proof of this theorem, we need a few lemmas.
The proof of the first preparatory lemma is a straightforward  extension of the proof of \cite[Theorem 4.4]{Hu1}.

\begin{Lemma}\label{lm:Huneke}
Let $G$ be a complete graph on the vertex set $[n]$ and $J_G$ its binomial edge ideal. Then $y_{n-2}-x_{n-1},y_{n-1}-x_n,y_n$ is a maximal regular sequence on 
$S/J_G^i$ for all $i\geq 2.$
\end{Lemma}

\begin{proof}
Let $R=S/(y_{n-2}-x_{n-1},y_{n-1}-x_n,y_n)=K[x_1,\ldots,x_{n-2},y_1,\ldots,y_{n-1}].$ The image of $J_G$ in $R$ is the ideal $J'$ generated by all the 
$2$--minors of the matrix \[X'=\left(
\begin{array}{ccccc}
	x_1& \ldots & x_{n-2} & y_{n-2} & y_{n-1}\\
	y_1 & \ldots & y_{n-2} & y_{n-1} & 0
\end{array}\right).
\]
In order to prove our claim, it is enough to show that $\mm,$ that is, the maximal ideal of $R$ is associated to $(J')^i$ for $i\geq 2.$ If we show that 
$((J')^i:y_{n-1}^{2i-1})$ is $\mm$--primary, then $\mm\in \Ass(R/(J')^i)$, which implies that $\depth(R/(J')^i)=0$ and this will finish the proof. 

Set $y=y_{n-1}.$ Then $y^{2i-1}\notin (J')^i$ since $(J')^i$ is generated in degree $2i.$ But $y\cdot y^{2i-1}=(y^2)^i\in (J')^i$ since $y^2\in J'.$ Therefore, 
$y\in (J')^i:y^{2i-1}.$ For $1\leq j\leq n-2,$ we have $y_jy\in J'.$ Then $y_j^{2i-1}y^{2i-1}\in (J')^{2i-1}\subseteq (J')^i,$ thus 
$y_j^{2i-1}\in J':y^{2i-1}$ for $1\leq j\leq n-2.$ Finally, since $x_jy-y_jy_{n-2}\in J'$ for $1\leq j\leq n-2,$ we get 
$y^{2i-2}(x_jy-y_jy_{n-2})\in (J')^{i-1}\cdot J'=(J')^i,$ since $y^{2i-2}=(y^2)^{i-1}\in (J')^{i-1}.$ On the other hand, 
$
y_jy_{n-2}y^{2i-2}=(y_jy)(y_{n-2}y)(y^2)^{i-2}\in (J')^i.
$ It follows that $x_jy^{2i-1}\in (J')^i,$ which implies that $x_j\in (J')^i:y^{2i-1}$ for $1\leq j\leq n-2.$
\end{proof}

\begin{Lemma}\label{lm:Huini}
In the settings of Lemma~\ref{lm:Huneke}, we have 
\[
\depth\frac{S}{\ini_<(J_G^i)} =3, 
\] for all $i\geq 2.$
\end{Lemma}

\begin{proof} By (\ref{eq:lemmaEH}), we have $\ini_<(J_G^i)=(\ini_<(J_G))^i$ for $i\geq 1.$ Since $y_1$ and $x_n$ form a regular sequence on $\ini_<(J_G),$ we get the following relations:
\[
\depth\frac{\overline{S}}{(\ini_<(J_G))^i}+2=\depth\frac{S}{(\ini_<(J_G))^i}=\depth\frac{S}{\ini_<(J_G^i)}\leq 
\depth\frac{S}{J_G^i}=3,
\] where $\overline{S}=K[\{x_i,y_j: 1\leq i\leq n-1, 2\leq j\leq n\}].$
By \cite[Theorem~4.4]{Trung}, $\depth\overline{S}/(\ini_<(J_G))^i\geq 1$ for $i\geq 1.$ Therefore, we get the desired equality.
\end{proof}

In the following lemma, we use the following notation. If $H$ is a graph on some vertex set $V(H),$ then we denote by $S(H)$ the polynomial ring over $K$ in the variables $x_k,y_k,$ where $k\in V(H).$

\begin{Lemma}\label{lm:regseq}
Let $G$ be a closed graph on the vertex set $[n]$ with the maxi\-mal cliques $[a_1,a_2],[a_2,a_3],$ $\ldots,[a_r,a_{r+1}]$ where $1=a_1<a_2<\cdots<a_r<a_{r+1}=n.$
Let $G'$ be the graph whose connected components are the mutually disjoint cliques \[[a_1,a_2],[a_2+1,a_3+1],\ldots,[a_r+(r-1),a_{r+1}+(r-1)].\] Then 
the following hold:
\begin{itemize}
	\item [\emph{(a)}] The sequence 
of linear forms 
\[
\underline{\ell}: \ell_1^y=y_{a_2}-y_{a_2+1},\ell_1^x=x_{a_2}-x_{a_2+1}, \ell_2^y=y_{a_3+1}-y_{a_3+2},\ell_2^x=x_{a_3+1}-x_{a_3+2}, \]
\[\ldots, 
\ell_{r-1}^y=y_{a_r+(r-2)}-y_{a_r+(r-1)},\ell_{r-1}^x=x_{a_r+(r-2)}-x_{a_r+(r-1)}
\] is regular on $S(G')/J_{G'}^j$ and 
\[
\frac{\frac{S(G')}{J_{G'}^j}}{(\underline{\ell})\frac{S(G')}{J_{G'}^j}}\cong \frac{S}{J_G^j}.
\] for every $j\geq 1.$
\item[\emph{(b)}] The sequence of variables
\[
\underline{\mu}: x_{a_2},y_{a_2+1},x_{a_3+1},\ldots,y_{a_r+(r-1)}
\] is regular on $S(G')/\ini_<(J_{G'}^j)$ and 
\[
\frac{\frac{S(G')}{\ini_<(J_{G'}^j)}}{(\underline{\mu})\frac{S(G')}{\ini_<(J_{G'}^j)}}\cong \frac{S}{\ini_<(J_G^j)}.
\] for every $j\geq 1.$
\end{itemize}
\end{Lemma}

\begin{proof} (a) Let $j\geq 1$ be an integer. 
We prove by induction on $2\leq i\leq r$ that the sequence $ \underline{\ell}_{i-1}:\ \ell_1^y,\ell_1^x,\ldots,\ell_{i-1}^y,\ell_{i-1}^x$ is  regular on $S(G')/J_{G'}^j$ and 
\[
\frac{\frac{S(G')}{J_{G'}^j}}{(\underline{\ell}_{i-1})\frac{S(G')}{J_{G'}^j}}\cong \frac{S(\widetilde{G}_{i-1})}{J_{\widetilde{G}_{i-1}}^j}.
\]
where, after relabeling   the vertices, $\widetilde{G}_{i-1}$ is a closed graph with  the maximal cliques $[a_1,a_2],[a_2,a_3],\ldots, [a_i,a_{i+1}]$,  
$[a_{i+1}+1,a_{i+2}+1],\ldots,[a_r+(r-i),a_{r+1}+(r-i)].$  

Let us first check the claim for $i=2.$ We have to show that $\ell_1^y,\ell_1^x$ is regular on $S(G')/J_{G'}^j.$ Note that 
$J_{G'}$ is a prime ideal since it is the sum of $r$ prime ideals in pairwise disjoint sets of variables corresponding to the $r$ connected components of $G'$; see \cite[Lemma 7.14]{HHO}.

Let $h\in S(G')$ such 
that $\ell_1^y h\in J_{G'}^j.$ Since $\ell_1^y\notin J_{G'},$ it follows that $h\in J_{G'}^{(j)}=J_{G'}^{j},$ thus $\ell_1^y$ is regular on $S(G')/J_{G'}^j.$ Now we show that $\ell_1^x$ is regular on $S(G')/(J_{G'}^j+(\ell_1^y)).$ We have 
\[
\frac{S(G')}{J_{G'}^j+(\ell_1^y)}\cong \frac{S(G')}{\overline{J}+(\ell_1^y)}
\] where $\overline{J} $ is the ideal in $S(G')$ generated by the polynomials $\overline{g}_1,\ldots, \overline{g}_m$ obtained 
from the generators $g_1,\ldots, g_m$ of $J_{G'}^j$ as follows. If $g_k$ is a generator which contains the variable $y_{a_2+1},$ we replace it by $y_{a_2}$ and denote the new binomial by $\overline{g}_k.$ If $g_k$ contains the variable $y_{a_2}$, we replace it by 
$y_{a_2+1},$ and denote the new binomial by $\overline{g}_k.$ Note that no generator of $J_{G'}^j$ contains both variables 
$y_{a_2}$  and $y_{a_2+1}$ since $\{a_2,a_2+1\}$ is not an edge in $G'.$ Finally, if $g_k$ is a generator of $J_{G'}^j$ which does not contain any of the variables $y_{a_2},y_{a_2+1}$, we simply set $\overline{g}_k=g_k.$ Then $\overline{g}_1,\ldots, \overline{g}_m$ are the generators of the $j^{th}$ power of the binomial edge ideal associated with the graph $G'$ and the matrix 
\[X'=\left(
\begin{array}{cccccccc}
x_1 & \cdots & x_{a_2-1} & x_{a_2} & x_{a_2+1} & x_{a_2+2} &\cdots & x_{a_{r+1}+r-1}\\
y_1 & \cdots & y_{a_2-1} & y_{a_2+1} & y_{a_2} & y_{a_2+2} &\cdots & y_{a_{r+1}+r-1}
\end{array}\right).
\] Since $G'$ consists of $r$ complete graphs, it follows that 
$\ini_<(\overline{J})$ is generated by the monomials  $\ini_<\overline{g}_1,\ldots, \ini_<\overline{g}_m$ where $<$ is the lexicographic order on 
$S(G').$ Note that 
$\ini_<\overline{g}_k$ differs from $\ini_<g_k$ if and only if $y_{a_2}|\ini_<g_k$ and, in this case, 
$\ini_<\overline{g}_k$ is obtained from $\ini_<g_k$ by replacing the variable $y_{a_2}$ with $y_{a_2+1}.$ Then it follows that 
none of the generators of the initial ideal of $\overline{J}+(\ell_1^y)$ is divisible by $x_{a_2}$ since $\{a_2,a_2+1\}$ is not an edge in $G'.$ Therefore, $x_{a_2}$ is regular on $\ini_<(\overline{J}+(\ell_1^y))$ and further, $x_{a_2}-x_{a_2+1}$ is regular on $S(G')/(J_{G'}^j+(\ell_1^y))$. Moreover, we get 
\[
\frac{S(G')}{J_{G'}^j+(\ell_1^y,\ell_1^x)}\cong \frac{S(\widetilde{G}_1)}{J_{\widetilde{G}_1}^j}
\]
where $\widetilde{G}_1$ is obtained from $G'$ by identifying the vertices $a_2$ and $a_2+1$ and by relabeling the vertices 
$k$ with $k-1$ for $k\geq a_2+2.$ Thus $\widetilde{G}_1$ has  the maximal cliques $[a_1,a_2],[a_2,a_3], [a_3+1,a_4+1],\ldots 
[a_r+(r-2),a_{r+1}+(r-2)].$ In particular, $\widetilde{G}_1$ is a closed graph which has $r-1$ connected components.

Assume that the sequence $ \underline{\ell}_{i-1}:\ \ell_1^y,\ell_1^x,\ldots,\ell_{i-1}^y,\ell_{i-1}^x$ is a regular sequence on $S(G')/J_{G'}^j$ and 
\[
\frac{\frac{S(G')}{J_{G'}^j}}{(\underline{\ell}_{i-1})\frac{S(G')}{J_{G'}^j}}\cong \frac{S(\widetilde{G}_{i-1})}{J_{\widetilde{G}_{i-1}}^j},
\]
where  the graph $\widetilde{G}_{i-1}$ has the first connected component consisting of the maximal cliques 
$[a_1,a_2],[a_2,a_3],\ldots, [a_i,a_{i+1}]$ and the other connected components are disjoint cliques. We have to show that 
$\ell_i^y,\ell_i^x$ is a regular sequence on $S(\widetilde{G}_{i-1})/J_{\widetilde{G}_{i-1}}^j.$ 

In the closed graph $\widetilde{G}_{i-1},$ the vertices
$a_{i+1}$ and $a_{i+1}+1$ are simplicial, thus $\ell_i^y$ does not belong to any minimal prime ideal of $\widetilde{G}_{i-1}$. This implies that $\ell_i^y$ is regular on $S(\widetilde{G}_{i-1})/J_{\widetilde{G}_{i-1}}^j$ since $J_{\widetilde{G}_{i-1}}^j$ has no embedded component by (\ref{eq:eqpowers}). It remains to show that $\ell_i^x$ is regular on $S(\widetilde{G}_{i-1})/(J_{\widetilde{G}_{i-1}}^j+(\ell_i^y)).$ The argument is very similar to the first step of the induction. We observe that $J_{\widetilde{G}_{i-1}}^j+(\ell_i^y)=\overline{J}+(\ell_i^y)$ where 
$\overline{J}$ is obtained as follows. Let $g_1,\ldots,g_m$ be the generators of $J_{\widetilde{G}_{i-1}}^j$  and denote by 
$\overline{g}_1,\ldots,\overline{g}_m$ the polynomials obtained in the following way. If $g_k$ contains the variable $y_{a_{i+1}},$ 
(respectively 
$y_{a_{i+1}+1}$) we replace 
it by $y_{a_{i+1}+1}$ (respectively by $y_{a_{i+1}}$),  and set $\overline{g}_k$ for the new binomial. Finally, if $g_k$ does not contain any of the variables $y_{a_{i+1}},$ $y_{a_{i+1}+1}$, we simply set $\overline{g}_k=g_k.$ Then  
$\overline{J}=(\overline{g}_1,\ldots,\overline{g}_m)$ is the $j^{th}$ power of the binomial edge ideal corresponding to the closed  graph 
$\widetilde{G}_{i-1}$ and the matrix
\[X'=\left(
\begin{array}{cccccccc}
x_1 & \cdots & x_{a_{i+1}-1} & x_{a_{i+1}} & x_{a_{i+1}+1} & x_{a_{i+1}+2} &\cdots & x_{a_{r+1}+r-i}\\
y_1 & \cdots & y_{a_{i+1}-1} & y_{a_{i+1}+1} & y_{a_{i+1}} & y_{a_{i+1}+2} &\cdots & y_{a_{r+1}+r-i}
\end{array}\right).
\] It follows that the initial ideal of $\overline{J}$ is minimally generated by the monomial generators  of 
$\ini_<(J_{\widetilde{G}_{i-1}}^j)$ in which we replaced the variable $y_{a_{i+1}} $ with $y_{a_{i+1}+1}.$ Hence 
$\overline{g}_1,\ldots,\overline{g}_m, \ell_i^y$ is a Gr\"obner basis of $\overline{J}+(\ell_i^y).$ This implies that all the monomial 
minimal generators of $\ini_<(\overline{J}+(\ell_i^y))$ are not divisible by $x_{a_{i+1}}.$ Therefore, $x_{a_{i+1}}$ is regular on 
$\ini_<(\overline{J}+(\ell_1^y))$ and, consequently, $\ell_i^x$ is regular on $S/(\overline{J}+(\ell_i^y)).$ Moreover, we  get the following isomorphism:
\[
\frac{S(\widetilde{G}_{i-1})}{J_{\widetilde{G}_{i-1}}^j}\cong \frac{S(\widetilde{G}_{i})}{J_{\widetilde{G}_{i}}^j}
\] where $\widetilde{G}_i$ is a closed graph which is obtained from $\widetilde{G}_{i-1}$ by identifying the vertex $a_{i+1}+1$ with 
$a_{i+1}$ and by relabeling the vertex $k$ with $k-1$ for $k\geq a_{i+1}+2.$ Thus, the new graph $\widetilde{G}_{i}$ has the maximal cliques
\[
[a_1,a_2],\ldots,[a_i,a_{i+1}], [a_{i+1},a_{i+2}], [a_{i+2}+1, a_{i+3}+1],\ldots,[a_r+(r-i-1),a_{r+1}+(r-i-1)].
\]
Therefore, the proof by induction is completed. 

(b) Since the variables from $\underline{\mu}$ do not appear in the support of the minimal generators of $\ini_<(J_{G'}),$ it obviously follows that
$\underline{\mu}$ is a regular sequence on $S'/(\ini_<(J_{G'}))^j=S'/\ini_<(J_{G'}^j)$ and the desired conclusion follows.
\end{proof}

\begin{Lemma}\label{lm:depthdisconnect}
Let $G'$ be the graph with the connected components $H_1, H_2,\ldots,H_r$, where each $H_i$ is a complete graph with $d_i+1$ vertices. Assume that 
$d_1\geq d_2\geq \cdots \geq d_r\geq 1.$ Let $J_{G'}$ be the binomial edge ideal of $G'$ in the polynomial ring $S'=K[\{x_i,y_i:i\in V(G')\}].$ Then: 
\begin{itemize}
	\item [\emph{(a)}] \[
	\depth\frac{S'}{J_{G'}^i}=\depth\frac{S'}{\ini_<(J_{G'}^i)}=d_i+d_{i+1}+\cdots +d_r+2r+i-1, \text{ for } 1\leq i\leq r,
	\]
	\item [\emph{(b)}] \[
	\depth\frac{S'}{J_{G'}^i}=\depth\frac{S'}{\ini_<(J_{G'}^i)}=3r, \text{ for } i\geq r+1.
	\]
\end{itemize}
\end{Lemma}

\begin{proof}
We proceed by induction on $i.$ To simplify the notation, we set $J_i=J_{H_i}$ for $1\leq i\leq r.$ For $i=1,$ we have 
\[ 
\depth\frac{S'}{J_{G'}}=\depth\frac{S_1}{J_1}+\cdots +\depth\frac{S_r}{J_r} \] and 
\[
\depth\frac{S'}{\ini_<(J_{G'})}=\depth\frac{S_1}{\ini_<(J_1)}+\cdots +\depth\frac{S_r}{\ini_<(J_r)}
\]  
where $ S_i=K[\{x_j,y_j:j\in V(H_i)\}] $ for $ 1\leq i\leq r.$

Since $J_i$ and $\ini_<(J_i)$ are Cohen-Macaulay for all $i,$ we get 
\[
\depth\frac{S'}{J_{G'}}=\depth\frac{S'}{\ini_<(J_{G'})}=(d_1+2)+(d_2+2)+\cdots+(d_r+2) =d_1+d_2+\cdots+d_r+2r.
\] 

The inductive step follows from the same arguments for $\depth S'/J_{G'}^i$ and  for $\depth S'/\ini_<(J_{G'}^i).$ We will explain in detail the proof for  $\depth S'/J_{G'}^i$ and, in the final part we will point out what is different in the proof for 
$\depth S'/\ini_<(J_{G'}^i).$

Let us assume that 
\[
\depth\frac{S'}{J_{G'}^i}=d_i+d_{i+1}+\cdots +d_r+2r+i-1
\] and 
\[
\depth\frac{S'}{\ini_<(J_{G'}^i)}=\depth\frac{S'}{\ini_<(J_{G'})^i}=d_i+d_{i+1}+\cdots +d_r+2r+i-1
\]
for $i\leq r-1.$

By \cite[Theorem 3.3]{HTT}, we have
\begin{equation}\label{eq:HTT}
\depth\frac{J_{G'}^i}{J_{G'}^{i+1}}=\min_{ j_1+j_2+\cdots+j_r=i}\left\{\depth\frac{J_1^{j_1}}{J_1^{j_1+1}}+ \depth\frac{J_2^{j_2}}{J_2^{j_2+1}}+\cdots+
\depth\frac{J_r^{j_r}}{J_r^{j_r+1}}\right\}.
\end{equation}

We know that  $\depth\frac{S_j}{J_i}=d_i+2\geq 3$ since $J_i$ is Cohen-Macaulay, and $\depth\frac{J_i^j}{J_i^{j+1}}=3$ for $j\geq 1$, by Lemma~\ref{lm:Huneke}.

If $i\leq r-1,$ in the equality $j_1+j_2+\cdots +j_r=i,$ at most $i$ exponents among $j_1,j_2,\ldots,j_r$ are not $0.$ Since $d_1\geq d_2\geq \cdots \geq d_r,$ we get 
\[
\sum_{s=1}^r\depth\frac{(J_i)^{j_s}}{(J_i)^{j_s+1}}\geq 3i+(d_{i+1}+2)+\cdots+ (d_r+2)=d_{i+1}+\cdots+d_r+2r+i.
\]
Moreover, the minimal value $d_{i+1}+\cdots+d_r+2r+i$ is taken for the exponents $j_1=\cdots=j_i=1$ and $j_{i+1}=\cdots =j_r=0.$ Hence,
equality (\ref{eq:HTT}) implies that 
\[
\depth\frac{J_{G'}^i}{J_{G'}^{i+1}}=d_{i+1}+\cdots+d_r+2r+i.
\] We have the exact sequence of $S'$--modules:
\[
0\to \frac{J_{G'}^i}{J_{G'}^{i+1}}\to \frac{S'}{J_{G'}^{i+1}}\to \frac{S'}{J_{G'}^{i}}\to 0.
\]
By the inductive hypothesis, since $i\leq r-1,$ we have 
$\depth \frac{S'}{J_{G'}^{i}}=d_i+d_{i+1}+\cdots+ d_r+2r+(i-1).$ As $d_{i+1}+\cdots+d_r+2r+i\leq d_i+d_{i+1}+\cdots+ d_r+2r+(i-1),$ by Depth Lemma applied to the above exact sequence, it follows that $\depth \frac{S'}{J_{G'}^{i+1}}=d_{i+1}+\cdots+d_r+2r+i.$ Therefore, we proved part (a) of the statement. In particular,
for $i=r,$ we have $\depth  \frac{S'}{J_{G'}^{r}}=d_r+3r-1.$ For proving part (b), we apply again induction on $i\geq r+1.$
We have the exact sequence of $S'$--modules:
\[
0\to \frac{J_{G'}^r}{J_{G'}^{r+1}}\to \frac{S'}{J_{G'}^{r+1}}\to \frac{S'}{J_{G'}^{r}}\to 0.
\] In equality (\ref{eq:HTT}), if we consider $j_1+j_2+\cdots +j_r=r$, we derive that 
\[
\sum_{s=1}^r\depth\frac{(J_i)^{j_s}}{(J_i)^{j_s+1}}\geq 3r
\] and the minimal value $3r$ is taken for $j_1=j_2=\cdots=j_r=1.$ Thus, $\depth \frac{J_{G'}^r}{J_{G'}^{r+1}}=3r.$ Since 
$3r\leq d_r+3r-1,$ Depth Lemma on the above exact sequence yields $\depth \frac{S'}{J_{G'}^{r+1}}=3r.$ For the inductive step, we consider the exact sequence 
\[
0\to \frac{J_{G'}^i}{J_{G'}^{i+1}}\to \frac{S'}{J_{G'}^{i+1}}\to \frac{S'}{J_{G'}^{i}}\to 0
\] for $i\geq r+1.$ By hypothesis we have $\depth \frac{S'}{J_{G'}^{i}}=3r,$ and we know from equality (\ref{eq:HTT}) that 
$\depth \frac{J_{G'}^i}{J_{G'}^{i+1}}=3r.$ Then, by Depth Lemma, we obtain $\depth \frac{S'}{J_{G'}^{i+1}}=3r.$

As we have already mentioned, the inductive step for  of $\depth S'/\ini_<(J_{G'}^i)$ works in the same way. The only difference is that we need to apply Lemma~\ref{lm:Huini} in order to derive that $\depth (\ini_<(J_i))^j/(\ini_<(J_i))^{j+1}=3$ for 
$j\geq 1.$

\end{proof}

\begin{proof}[Proof of Theorem~\ref{thm:powersclosed}] To begin with, we prove the formulas for the depth of $S/J_G^i.$
Let $[a_1,a_2],[a_2,a_3],\ldots,[a_r,a_{r+1}]$ be the maximal cliques of $G,$ where $1=a_1<a_2<\cdots <a_r<a_{r+1}=n.$ Note that this is not necessarily the order with respect to the dimensions of the cliques. Let $G'$ be the graph on $[n+r-1]$  with the connected components 
$[a_1,a_2],[a_2+1,a_3+1],\ldots,[a_r+(r-1),a_{r+1}+(r-1)]$ and $J_{G'}\subset S'=K[\{x_j,y_j:j\in v(G')\}]$ the associated binomial edge ideal. 
By Lemma~\ref{lm:depthdisconnect}, we have 
\[
\depth\frac{S'}{J_{G'}^i}=\depth\frac{S'}{\ini_<(J_{G'}^i)}=\left\{
\begin{array}{ll}
d_i+d_{i+1}+\cdots +d_r+2r+(i-1), & \text{ for } 1\leq i\leq r,\\
3r, & \text{ for }  i\geq r+1.
\end{array}\right.
\] By Lemma~\ref{lm:regseq}, the sequence of  $2(r-1)$ linear forms 
\[
\underline{\ell}: \ell_1^y=y_{a_2}-y_{a_2+1},\ell_1^x=x_{a_2}-x_{a_2+1}, \ell_2^y=y_{a_3+1}-y_{a_3+2},\ell_2^x=x_{a_3+1}-x_{a_3+2}, \]
\[\ldots, 
\ell_{r-1}^y=y_{a_r+(r-2)}-y_{a_r+(r-1)},\ell_{r-1}^x=x_{a_r+(r-2)}-x_{a_r+(r-1)}
\] is regular on $S'/J_{G'}^i$ and $S'/(J_{G'}^i+(\underline{\ell}))\cong S/J_G^i$ for all $i\geq 1.$ In addition, the sequence 
\[
\underline{\mu}: x_{a_2},y_{a_2+1},x_{a_3+1},\ldots,y_{a_r+(r-1)}
\] is regular on $S(G')/\ini_<(J_{G'}^j)$ and 
\[
\frac{\frac{S(G')}{\ini_<(J_{G'}^j})}{(\underline{\mu})\frac{S(G')}{\ini_<(J_{G'}^j)}}\cong \frac{S}{\ini_<(J_G^j)}.
\] for every $j\geq 1.$
This implies that 
\[
\depth\frac{S}{J_{G}^i}=\depth \frac{S}{\ini_<(J_G^j)}=\]  \[=\left\{
\begin{array}{ll}
\sum_{j=i}^r d_j+i+1=n-d_1-d_2\cdots-d_{i-1}+i, & \text{ for } 1\leq i\leq r,\\
r+2, & \text{ for }  i\geq r+1.
\end{array}\right.
\]
\end{proof}

With similar arguments as we used for the connected case, we may derive the depth function for the powers of $J_G$ and 
$\ini_<(J_G)$ in the case that $G$ has several connected components, say $G_1,\ldots,G_c.$ The only difference is that we do not need to mod out the entire sequences $\underline{\ell}$ and $\underline{\mu},$ but, instead, sequences of length $2(r-1)-2(c-1)=2(r-c).$  Consequently, we get  the following.

\begin{Proposition}\label{pr:discdepth}
Let $G$ be a closed graph on the vertex set $[n]$  with the connected components $G_1,G_2,\ldots,G_c$ such that $J_G$ is Cohen-Macaulay. Let $F_1,F_2,\ldots,F_r$ be the maximal cliques of 
$G$ and $d_i=\dim F_i=\# F_i-1$ for $1\leq i\leq r.$ Assume that $d_1\geq d_2\geq \cdots\geq d_r\geq 1.$ Then: 
\begin{itemize}
	\item [\emph{(a)}] \[\depth\frac{S}{J_G^i}=\depth\frac{S}{\ini_<(J_G^i)}=n-\sum_{j=1}^{i-1}d_j+i+c-1, \text{ for }1\leq i\leq r,\]
	\item [\emph{(b)}] \[\depth\frac{S}{J_G^i}=\depth\frac{S}{\ini_<(J_G^i)}=r+2c, \text{ for } i\geq r+1.\]
\end{itemize}
\end{Proposition}


\begin{Proposition}\label{pr:closednotCM}
Let $G$ be a closed graph with the property that at least one  of its connected components is not a path. Then $J_G^i$ is not Cohen-Macaulay for $i\geq 2.$
\end{Proposition}

\begin{proof} 
If $J_G$ is Cohen-Macaulay, then, by Proposition~\ref{pr:discdepth}, it follows that \[\depth(S/J_G^i)<\depth(S/J_G)=\dim(S/J_G)\] for 
$i\geq 2$  since $G$ has cliques with at least $3$ vertices. This implies that $J_G^i$ is not Cohen-Macaulay. 

If $J_G$ is not Cohen-Macaulay, then, by Theorem~\ref{Th:classification}, $J_G$ is not unmixed. This implies that $J_G^{(i)}$ is not unmixed, thus it is not 
Cohen-Macaulay. But we know  that $J_G^i=J_G^{(i)}$ for all $i\geq 1,$ therefore $J_G^i$ is not Cohen-Macaulay for $i\geq 1.$ 
\end{proof}

Since all the powers of a complete intersection ideal in a polynomial ring are Cohen-Macaulay \cite{AV,CN,W}, we get the following consequence of the above proposition. 

\begin{Corollary}\label{cor:pathpowers}
Let $G$ be a closed graph. Then the following are equivalent:
\begin{itemize}
	\item [(a)] Each connected component of $G$ is a path graph,
	\item [(b)] $J_G^i$ is Cohen-Macaulay for every $i\geq 2,$
	\item [(c)] $J_G^i$ is Cohen-Macaulay for some $i\geq 2,$
	\item [(c)] $J_G^2$ is Cohen-Macaulay.
\end{itemize}
\end{Corollary}

\begin{Proposition}\label{pr:Rees}
Let $G$ be a closed graph and let $J_G$ be the associated binomial edge ideal. Then the Rees algebras $\MR(J_G)$ and 
$\MR(\ini_<(J_G))$ are Cohen-Macaulay and have the same dimension. In particular, the graded rings of $J_G$ and $\ini_<(J_G)$ are Cohen-Macaulay.
\end{Proposition}

\begin{proof} Since $\ini_<(J_G)$ is normally torsion free, it follows that $\MR(\ini_<(J_G))$ is Cohen-Macaulay by \cite{Ho} and,  by 
\cite[Theorem~2.7]{CHV}, we have $\MR(\ini_<(J_G))=\ini_{<'}(\MR(J_G)).$ Here $\ini_{<'}(\MR(J_G))$ is the initial algebra of $\MR(J_G)$ with respect to the monomial order $<'$ on $S[t]$ which extends the lexicographic order $<$ on $S$  as follows: 
given two monomials $u,v\in S$ and two integers $i,j\geq 0,$ we have $ut^i<vt^j$ if and only if $i<j$ or $i=j$ and $u<v.$ Since 
$\MR(\ini_<(J_G))$ is Cohen-Macaulay,  it follows that  $\ini_{<'}(\MR(J_G))$ is Cohen-Macaulay and this implies that $\MR(J_G)$ shares the same property \cite{CHV}. In addition, as $\ini_{<'}(\MR(J_G))$ and $\MR(J_G)$ have the same Krull dimension \cite{CHV}, it follows that 
$\MR(J_G)$ and $\MR(\ini_<(J_G))$ have the same dimension.

The last part of the statement follows by \cite[Proposition 1.1]{Hu2}.
\end{proof}

Theorem~\ref{thm:powersclosed} shows that the depth function of  Cohen-Macaulay binomial edge ideals of closed graphs is 
non-increasing. Moreover, it coincides with the depth function of their initial ideals. We expect that this behavior holds for every closed graph, but we could not prove it. Instead, in the next theorem  we show that, for every closed graph $G,$ the ideals $J_G$ and $\ini_<(J_G)$ 
have the same depth limit and we compute its value. Moreover, in Proposition~\ref{pr:nondecr}, we will show that $\ini_<(J_G)$
has a non-increasing depth function. 

Before stating the theorem, let us recall a few notions and results. A classical result of Brodmann \cite{Brod} states that  if $I$ is a homogeneous ideal in a polynomial ring 
$R=K[x_1,\ldots,x_n]$,  then
\begin{equation}\label{eq:limde}
\lim_{k\to\infty}\depth\frac{R}{I^k}\leq n-\ell(I),
\end{equation}
 where $\ell(I)=\dim \MR(I)/\mm \MR(I)$ is the 
analytic spread of $I.$ Here $\mm=(x_1,x_2,\ldots,x_n)$ is the maximal graded ideal of $R$ and $\MR(I)$ is the Rees algebra of the ideal $I.$ For an alternative proof of (\ref{eq:limde}) we refer to \cite[Theorem~1.2]{HeHi}. In \cite{EiHu}, it was shown that the equality holds in (\ref{eq:limde}) if the ring $\gr_I(R)$ is Cohen-Macaulay, which is the case if $\MR(I)$ is Cohen-Macaulay \cite{Hu2}.
We should also recall that if $I$ is generated by some polynomials, say $f_1,\ldots, f_m,$ of the same degree, than the fiber ring 
$\MR(I)/\mm \MR(I)$ is equal to $K[f_1,\ldots,f_m].$

On the other hand, we need to recall some graph theoretical  terminology.  A vertex $v$ of the graph $G$ is called a \emph{free} vertex if it belongs to exactly one maximal cligue of $G.$ A connected graph $G$ is called \emph{decomposable} if there exists  $G_1$ and $G_2$ subgraphs of 
$G$ such that $G=G_1\cup G_2$ with $V(G_1)\cap V(G_2)=\{v\}$ and $v$ is a free vertex in $G_1$ and $G_2.$ A connected graph $G$ is \emph{indecomposable} if it is not decomposable. Clearly, every graph $G$ (not necessarily connected) has a unique decomposition up to ordering  of the form 
$G=G_1\cup G_2\cup\cdots \cup G_r$ where $G_1,\ldots,G_r$ are indecomposable graphs and for every $1\leq i<j\leq r,$ we have either 
$V(G_i)\cap V(G_j)=\emptyset$ or $V(G_i)\cap V(G_j)=\{v\}$ where $v$ is a free vertex in $G_i$ and $G_j.$ We call 
$G_1,\ldots, G_r$ the \emph{indecomposable components} of $G.$

\begin{Theorem}\label{th:limdepth}
Let $G$  be a closed graph and $J_G\subset S$ its binomial edge ideal. Let $g_1,\ldots, g_m$ be the generators of $J_G.$ Then the following hold:
\begin{itemize}
	\item [\emph{(a)}] The set $\{g_1,\ldots, g_m\}$ is a Sagbi basis of the $K$-algebra $K[g_1,\ldots, g_m]$ with respect to the lexicographic order on $S,$ that is, \[\ini_<(K[g_1,\ldots, g_m])=K[\ini_<g_1,\ldots,\ini_<g_m].\]
	\item [\emph{(b)}] The ideals $J_G$ and $\ini_<(J_G)$ have the same analytic spread.
	\item [\emph{(c)}] \[\lim_{k\to\infty}\depth\frac{S}{J_G^k}=\lim_{k\to\infty}\depth\frac{S}{(\ini_<(J_G))^k}=r+2,\] where $r$ is the number of indecomposable components of $G.$
\end{itemize}
\end{Theorem}

\begin{proof}
Let $A=K[g_1,\ldots, g_m]$ and $B=K[\ini_<g_1,\ldots,\ini_<g_m].$ 

(a). In order to show that $\{g_1,\ldots, g_m\}$ is a Sagbi basis of $A,$ 
we apply a criterion which plays a  similar role to the Buchberger criterion in the Gr\"obner basis theory; see 
\cite[Theorem 6.43]{EHbook}. Let $
\varphi:K[t_1,\ldots,t_m]\to A$ and 
$\psi:K[t_1,\ldots,t_m]\to B$ be the $K$-algebra homomorphisms defined by $\varphi(t_i)=g_i$ and $\psi(t_i)=\ini_<g_i$ for 
$1\leq i\leq m.$ Let $\tb^{\ab_1}-\tb^{\bb_1},\ldots,\tb^{\ab_r}-\tb^{\bb_r}$ be a system of binomial  generators for the toric ideal 
$\ker \psi.$ Then $\{g_1,\ldots, g_m\}$ is a Sagbi basis of $A$ if and only if there exist some coefficients $c_{\ab}^{(j)}\in K$ such 
that 
\[
\gb^{\ab_j}-\gb^{\bb_j}=\sum_{\ab}c_{\ab}^{(j)} \gb^{\ab}
\] with $\ini_<(\gb^{\ab})<\ini_<(\gb^{\ab_j})$ for all $\ab,$ where by $\gb^\ab$ we mean $g_1^{a_1}\cdots g_m^{a_m}$ if 
$\ab=(a_1,\ldots,a_m).$ Thus, we first need to find a set of binomial generators for $\ker \psi. $ The $K$-algebra $B$ is the edge ring 
of the bipartite graph $H$ on the vertex set $V(H)=\{x_1,\ldots,x_n\}\cup\{y_1,\ldots, y_n\}$ and edge set $E(H)=\{\{x_i,y_j\}:i<j \text{
 and }\{i,j\}\in E(G)\}.$ By \cite[Lemma~3.3]{EZ}, we know that every induced cycle in $H$ has length $4.$ By \cite{OH}, the toric 
ideal of $B$ is generated by the binomials $\beta_{\gamma_1},\ldots,\beta_{\gamma s}$ where $\gamma_1,\ldots, \gamma_s$ are the $4$-
cycles of $H$. If $\gamma$ is a $4$-cycle in $H,$ say $\gamma=(x_i,y_j, x_k,y_\ell)$ with $i<k<j<\ell,$ and 
$x_iy_j=\ini g_{i_1}, x_iy_\ell=\ini_<g_{i_2}, x_ky_j=\ini_<g_{i_3},x_ky_\ell=\ini_<g_{i_4},$ then 
$\beta_\gamma=t_{i_1}t_{i_4}-t_{i_2}t_{i_3}$. We have to lift the relations determined by the binomials $\beta_\gamma$ to $A.$ But this is very easy since
\[
g_{i_1}g_{i_4}-g_{i_2}g_{i_3}=g_{i_5}g_{i_6},
\] where $g_{i_5}=x_iy_k-x_ky_i$ and $g_{i_6}=x_jy_\ell-x_\ell y_j.$ Note that since $i<k<j<\ell,$ and $\{i,\ell\}\in E(G),$ then 
$\{i,k\}$ and $\{j,\ell\}$ are edges in $G$ as well, by Theorem~\ref{th:closedchar} (iv). Moreover, 
\[\ini_<(g_{i_5}g_{i_6})=x_ix_jy_ky_\ell<x_i x_ky_j y_\ell=\ini_<(g_{i_1}g_{i_4})\] since $k<j.$ Therefore, the proof of (a) is completed. 

(b) follows from (a) since $\dim A=\dim \ini_<A$ by \cite[Proposition 2.4]{CHV}. 

(c) By \cite[Proposition~3.3]{EiHu}, we have 
\[
\lim_{k\to\infty}\depth\frac{S}{J_G^k}=\dim S-\ell(J_G) \text{ and } \lim_{k\to\infty}\depth\frac{S}{(\ini_<(J_G))^k}=\dim S-\ell(\ini_<(J_G)).
\] Therefore, we get the equality of the two limits by (b).

Since $y_1$ and $x_n$ are isolated vertices in the bipartite graph $H$ whose edge ideal is equal to $\ini_<(J_G),$ we have
\[
\lim_{k\to\infty}\depth\frac{S}{(\ini_<(J_G))^k}=\lim_{k\to\infty}\depth\frac{S'}{I(H)^k}+2,
\] where $S'$ is the polynomial ring in the variables $x_j, 1\leq j\leq n-1$ and $y_j, 2\leq j\leq n.$ By \cite[Theorem~4.4]{Trung} or 
\cite[Corollary 10.3.18]{HH10}, 
\[
\lim_{k\to\infty}\depth\frac{S'}{I(H)^k}=r,
\] where $r$ is the number of connected components of $H.$ But, taking into account the characterization of closed graphs given in
 Theorem~\ref{th:closedchar} (iii),   it is easily seen that this is exactly the number of indecomposable components of $G.$
\end{proof}

\begin{Remark} By using \cite[Theorem 4.6 and Corollary 4.9]{ALL}, one may derive that the depth limit coincides for the closed determinantal facet ideals and their initial ideals with respect to the lexicographic order. This class of ideals was introduced in 
\cite{EHHM}.  
\end{Remark}

\begin{Proposition}\label{pr:nondecr}
Let $G$ be a closed graph and $J_G$ its binomial edge ideal. Then 
\[
\depth\frac{S}{(\ini_<(J_G))^{k+1}}\leq \depth\frac{S}{(\ini_<(J_G))^k}
\] for every $k\geq 1.$
\end{Proposition}

\begin{proof}
The inequalities follow by \cite[Theorem 5.2]{KTY} since the bipartite graph $H$ whose edge ideal is equal to $\ini_<(J_G)$ has at least 
one leaf, namely the edge $x_{n-1}y_n.$
\end{proof}

\medskip

As we have seen in Section~\ref{S:Prelim}, for every closed graph $G,$ we have $J_G^i=J_G^{(i)}$ for $i\geq 1.$ This equalities imply
that $\Ass(J_G^i)$ is constant for $i\geq 1,$ thus $J_G$ has the persistence property. But we can prove even more, namely, that 
$J_G$ has the strong persistence property. Let us recall that an ideal $I$ in a polynomial ring satisfies the strong persistence property if and only if $I^{k+1}:I=I^k$ for all $k;$ see \cite{HQ}.  We will derive this property from a slightly more general statement.

\begin{Proposition}\label{pr:strongpersist}
Let $I\subset R=K[x_1,x_2,\ldots,x_n]$ be a homogeneous ideal and assume that there exists a monomial order $<$ on $R$ such that the following conditions hold:
\begin{itemize}
	\item [(a)] $\ini_<(I)$ has the strong persistence property,
	\item [(b)] $\ini_<(I^j)=(\ini_<(I))^j$ for every $j\geq 1.$
\end{itemize}
Then the ideal $I$ has the strong persistence property. In particular, $I$ has the persistence property, that is, $\Ass(I^{j+1})\subseteq \Ass(I^j)$ for every 
$j\geq 1.$
\end{Proposition}

\begin{proof}
We have to prove that $I^{j+1}:I=I^j$ for $j\geq 1.$ Since $I^j\subseteq I^{j+1}:I$, it is enough to show that $\ini_<(I^j)=\ini_<(I^{j+1}:I).$ The inclusion
$\ini_<(I^j)\subseteq\ini_<(I^{j+1}:I)$ is obvious. For the other inclusion, let us consider a monomial $w\in \ini_<(I^{j+1}:I)$. Then there exists a polynomial 
$g\in I^{j+1}:I$ such that $w=\ini_<(g).$ As $g I\subseteq I^{j+1}$, we get
\[
w \ini_<(I)\subseteq \ini_<(I^{j+1})=(\ini_<(I))^{j+1},
\] which yields 
\[
w\in (\ini_<(I))^{j+1}:\ini_<(I)=(\ini_<(I))^{j}=\ini_<(I^j).
\]
\end{proof}

\begin{Corollary}\label{cor:strongpersist}
Let $G$ be a closed graph. Then $J_G$ has the strong persistence property. 
\end{Corollary}

\begin{proof} Let $<$ be the lexicographic order on $S.$ Then $\ini_<(J_G)=(x_iy_j:\{i,j\}\in E(G))$ is an edge ideal. Therefore, by \cite[Lemma~2.12]{MMV},
it follows that $\ini_<(J_G)$ has the strong persistence property. Moreover, by (\ref{eq:lemmaEH}), we also have 
$\ini_<(J_G^i)=(\ini_<(J_G))^i$ for every $i\geq 1.$ Hence, we may apply Proposition~\ref{pr:strongpersist}.
\end{proof}

\section{Regularity}
\label{S:Reg}

In the following theorem we compute the regularity of the powers of binomial edge ideals associated with connected closed graphs and of their initial ideals. Before stating this result, we recall some notions and results of graph theory.

A graph $G$ is called \textit{co-chordal} if its complement graph $G^c$ is chordal. The \textit{co-chordal cover number} of $G$, denoted $\cochord(G),$ is the smallest number $m$ for which there exist some co-chordal subgraphs $G_1,\ldots,G_m$ of $G$ such that $E(G)=\cup_{i=1}^mE(G_i).$ 

A graph $G$ is \textit{weakly chordal} if every induced cycle in $G$ and in $G^c$ has length at most $4.$
For a graph $G,$ we denote by $\im(G)$ the number of edges in a largest induced matching
of $G$. By an \textit{induced matching} we mean an induced subgraph of $G$ which consists of pairwise disjoint edges. In other words, 
$\im(G)$ is the monomial grade of the edge ideal $I(G).$ In \cite[Proposition~3]{BDS} it is proved that if $G$ is weakly chordal, then $\im(G)=\cochord(G).$ 

On the other hand, we will use \cite[Theorem~3.6]{JNS} which states that if $H$ is a bipartite graph and $I(H)$ is its edge ideal, then, for $i\geq 1,$ we have
\begin{equation}\label{eq:JNS}
\reg(I(H)^i)\leq \cochord(H)+2i-1.
\end{equation}

\begin{Theorem}\label{thn:regclosed}
Let $G$ be a connected closed graph. Then, for every $i\geq 1,$ we have
\[
\reg\frac{S}{J_G^i}=\reg\frac{S}{\ini_<(J_G^i)}=\ell+2(i-1),
\] where $\ell$ is the length of the longest induced path in $G.$
\end{Theorem}

\begin{proof}
The inequality $\reg S/J_G^i\geq \ell+2(i-1)$ follows by \cite[Corollary 3.4]{JKS}. Hence, we have 
\[\reg\frac{S}{\ini_<(J_G^i)}\geq \reg\frac{S}{J_G^i}\geq \ell+2(i-1).
\]
Thus, it is enough to prove that $\reg S/\ini_<(J_G^i)\leq\ell+2(i-1).$ Since $J$ is closed, by (\ref{eq:lemmaEH}), we have $\ini_<(J_G^i)=(\ini_<(J_G))^i.$ Therefore, we get
\[ \reg\frac{S}{\ini_<(J_G^i)}=\reg\frac{S}{(\ini_<(J_G))^i}.
\] As we have already mentioned in Section~\ref{S:Prelim}, the monomial ideal $\ini_<(J_G)=(x_iy_j: \{i,j\}\in E(G))$ is the edge ideal
$I(H)$ of a bipartite graph on 
$\{x_1,x_2,\ldots,x_n\}\cup\{y_1,y_2,\ldots,y_n\}.$  Then inequality (\ref{eq:JNS}) implies that 
\[
\reg\frac{S}{(\ini_<(J_G))^i}\leq \cochord(H)+2(i-1).
\] 
In \cite[Lemma~3.3]{EZ} it was proved that $H$ is a weakly chordal graph. This implies that $\cochord(H)=\im(H).$ On the other hand, by 
\cite[Proposition~3.5]{EZ}, it follows that $\im(H)=\ell,$ which completes the proof.
\end{proof}

The arguments of the above proof can be extended to disconnected closed graphs. 

\begin{Proposition}\label{pr:regdisconnect}
Let $G$ be a closed graph with the connected components $G_1,\ldots,G_c$. Let $\ell_i$ be the length of the longest induced path in the component $G_i$ for $1\leq i\leq c.$ Then, for all $i\geq 1,$ we have
\[
\reg\frac{S}{J_G^i}=\reg\frac{S}{\ini_<(J_G^i)}=\ell_1+\ell_2+\cdots+\ell_c+2(i-1).
\]
\end{Proposition}

\begin{proof} The inequality $\reg S/J_G^i\geq \ell_1+\ell_2+\cdots+\ell_c+2(i-1)$ follows from \cite[Proposition~3.3]{JKS} and \cite[Observation~3.2]{JKS} since the union of the longest induced paths in $G_j,$  $1\leq j\leq c,$ form an induced subgraph in $G,$ and the inequality $\reg S/\ini_<(J_G^i)\leq\ell_1+\ell_2+\cdots+\ell_c+2(i-1)$ holds since, 
obviously, in the bipartite graph $H$ such that $\ini_<(J_G)=I(H)$ we have $\im(H)=\ell_1+\ell_2+\cdots+\ell_c.$
\end{proof}

\section{Powers of binomial edge ideals of block graphs}
\label{S:Block}

In this section we discuss powers of binomial edge ideals of block graphs. We recall that a graph $G$ is called a \emph{block graph}, if each block of $G$ is a clique. A block of $G$ is a connected subgraph of G that has no cutpoint and is maximal with
respect to this property. The block graphs whose binomial edge ideal is Cohen-Macaulay are classified in \cite[Theorem 1.1]{EHH}. It is shown that for a block graph $G,$ the following conditions are equivalent:
\begin{enumerate}
\item[(a)] $J_G$ is unmixed.
\item[(b)] $J_G$ is Cohen--Macaulay.
\item[(c)] Each vertex of $G$ is the intersection of at most two maximal cliques.
\end{enumerate}

The following theorem shows that the equality between symbolic and ordinary powers does not hold, in general, for binomial edge ideals of block graphs.

\begin{Theorem}\label{thm:netfree}
Let $G$ be a  block graph such that $J_G$ is Cohen-Macaulay. Then, the following statements are equivalent:
\begin{itemize}
	\item [\emph{(a)}] $G$ is closed,
	\item [\emph{(b)}] $J_G^i=J_G^{(i)}$ for all $i\geq 2,$
	\item [\emph{(c)}] $J_G^i=J_G^{(i)}$ for some $i\geq 2,$
	\item [\emph{(d)}] $J_G^2=J_G^{(2)},$
	\item [\emph{(e)}] $G$ is net-free, that is, it does not contain a net as an induced subgraph.
\end{itemize}
\end{Theorem}

\begin{proof}
(a)$\Rightarrow$(b) follows by (\ref{eq:eqpowers}). The implications (b)$\Rightarrow$(c) and  (b)$\Rightarrow$(d) are trivial. 

Next we prove (d)$\Rightarrow$(e). Let us assume that $G$ contains as an induced subgraph the net $N$ with the edge set  
$E(N)=\{\{1,2\},\{3,4\},\{5,6\},\{2,3\},\{3,5\},\{2,5\}\}.$

Set
$g=x_3x_5x_6y_1y_2y_4-x_1x_5x_6y_2y_3y_4-x_3x_4x_5y_1y_2y_6+x_1x_2x_5y_3y_4y_6+x_1x_3x_4y_2y_5y_6-x_1x_2x_3y_4y_5y_6$.
We show that $g\in J_G^{(2)}\setminus J_G^2.$

Since $J_G=\bigcap_{ P_W(G)\in  \Ass J_G}P_W(G),$ we show  $g\in (P_W(G))^2,$ for all $W$ with the property that $P_W(G)\in  \Ass J_G$.
Then it follows that $g\in J_G^{(2)}$, because, as it was observed in \cite{EH}, $P_W(G)^{(i)}=P_W(G)^i$ for all $i\geq 1$ and all 
$W\subset [n].$
We consider the following cases.

\emph{Case 1.} If $W \cap [6]$ is equal to  $\emptyset$ or $\{1\}$ or $\{4\}$ or $\{6\},$ then 
$g=x_5y_4(x_6y_2-x_2y_6)(x_3y_1-x_1y_3)+x_3y_6(x_4y_2-x_2y_4)(x_1y_5-x_5y_1) \in  (P_W(G))^2$.

\emph{Case 2.} $W \cap [6]= \{2\}.$ Then 
$g=x_1x_2y_4y_6(x_5y_3-x_3y_5)+x_5x_6y_1y_2(x_3y_4-x_4y_3)+x_3x_4y_1y_2(x_6y_5-x_5y_6)+x_4x_6y_1y_2(x_5y_3-x_3y_5)
+x_1x_5y_2y_3(x_4y_6-x_6y_4)+x_1x_4y_2y_6(x_3y_5-x_5y_3)\in  (P_W(G))^2$.

\emph{Case 3.} $W \cap [6]= \{3\}.$ Then 
$g=x_1x_3y_3y_4(x_2y_6-x_6y_2)+x_3x_4y_1y_2(x_5y_6-x_6y_5)+x_1x_3y_4y_6(x_5y_2-x_2y_5)+x_3x_5y_2y_4(x_1y_6-x_6y_1)
+x_3x_4y_2y_5(x_5y_6-x_6y_5)\in  P_W(G)^2$.

\emph{Case 4.} $W \cap [6]= \{5\}.$ Then 
$g=x_1x_3y_5y_6(x_4y_2-x_2y_4)+x_5x_6y_2y_4(x_3y_1-x_1y_3)+x_2x_5y_3y_6(x_1y_4-x_4y_1)+x_4x_5y_1y_6(x_2y_3-x_3y_2)\in  (P_W(G))^2$.

Next, we show $g \not\in J_G^{2}$. Since $N$ is an induced subgrph of $G$, by the proof of \cite[Proposition 3.3]{JKS}, it follows that 
$J_N^i=J_G^i\cap K[x_1,\ldots,x_6,y_1,\ldots,y_6]$ for all $i\geq 1.$  Therefore, it suffices to show that $g \not\in J_N^{2}$.

Suppose  $g \in J_N^{2}$.
Then we have $x_1x_2x_5y_3y_4y_6 \in J_N^{2}+(x_3, y_2)$.
Since $\{x_1, y_4\}$ is a regular sequence on $S/(J_N^{2}+(x_3, y_2))$,
we have  $x_2x_5y_3y_6 \in  J_N^{2}+(x_3, y_2)$.
Since any monomial of degree $4$ in $ J_N^{2}+(x_3, y_2)$ which is not divided by neither $x_3$ nor  $y_2$ is not divided by  $y_6$,
it follows   $x_2x_5y_3y_6  \not\in  J_N^{2}+(x_3, y_2)$,   contradiction.

For (c)$\Rightarrow$(e), we show that $g(x_2y_3-x_3y_2)^{i-2}\in J_G^{(i)} \setminus J_G^{i}.$ Taking into account the above arguments, it is obvious that $g(x_2y_3-x_3y_2)^{i-2}\in (P_W(G))^i,$ for all $W$ with the property that $P_W(G)\in  \Ass J_G$, thus 
$g(x_2y_3-x_3y_2)^{i-2}\in J_G^{(i)}.$
We show $g(x_2y_3-x_3y_2)^{i-2} \not\in J_N^{i}$.
Suppose   $g(x_2y_3-x_3y_2)^{i-2}  \in J_N^{i}$.
Then we have $x_1x_2x_5y_3y_4y_6 (x_2y_3)^{i-2} \in J_N^{i}+(x_3, y_2)$.
Since $\{x_1, y_4\}$ is a regular sequence on $S/(J_N^{i}+(x_3, y_2))$,
we have  $x_2x_5y_3y_6(x_2y_3)^{i-2} \in  J_N^{i}+(x_3, y_2)$.
Since any monomial of degree $2i$ in $ J_N^{i}+(x_3, y_2)$ which is not divided by neither $x_3$ nor  $y_2$ is not divided by  $y_6$, it follows that   $x_2x_5y_3y_6 (x_2y_3)^{i-2} \not\in  J_N^{i}+(x_3, y_2)$,  contradiction.

Finally, we show that (e)$\Rightarrow$(a). Since $G$ is a block graph, it follows that $G$ is chordal and tent-free. On the other hand, as $J_G$ is Cohen-Macaulay, in particular, unmixed, it follows that $G$ is claw-free. Therefore, the hypothesis implies that $G$ is closed by Theorem~\ref{th:closedchar} (v).
\end{proof}

\begin{Proposition}\label{pr:nonCMblock}
Let $G$ be a connected block graph which is not a path. Then $J_G^i$ is not Cohen-Macaulay for every $i\geq 2.$
\end{Proposition}

\begin{proof}
We analyze the following cases.

\emph{Case 1.} Suppose that $G$ is a  net-free block graph which is not a path and $J_G$ is Cohen-Macaulay. Then, by using 
Theorem~\ref{th:closedchar} (v), it follows that 
$G$ is a closed graph. Then,  Proposition~\ref{pr:closednotCM} implies that $J_G^i$ is not Cohen-Macaulay for every $i\geq 2.$

\emph{Case 2.} Let $G$ be a block graph which contains a net as an induced subgraph and  such that $J_G$ is Cohen-Macaulay. Then, by Theorem~\ref{thm:netfree}, 
we have $J_G^i\subsetneq J_G^{(i)},$ for every $i\geq 2.$ In particular, it follows that $J_G^i$ has embedded components. Consequently, $J_G^i$ is not unmixed,
and, therefore, $J_G^i$ is not Cohen-Macaulay for $i\geq 2.$

\emph{Case 3.} Suppose that $G$ is a block graph and $J_G$ is not Cohen-Macaulay. Then $J_G$ is not unmixed. It follows that 
$J_G^{i}$ is not unmixed for all $i,$ thus  $J_G^i$ is not Cohen-Macaulay as well. 
\end{proof}

\section{Open Problems}
\label{S:Open}

As we have seen in Section~\ref{S:Depth}, the depth function of Cohen-Macaulay binomial edge ideals of closed graphs is non-increasing. 
The depth function of $\ini_<(J_G)$ is also non-increasing for every closed graph $G.$ Therefore it is natural to ask the following.

\begin{Question}\label{Q1}
Is it true that the depth function of $J_G$ is non-increasing for every closed graph $G?$ 
\end{Question}

Of course, taking into account Proposition~\ref{pr:nondecr}, we can answer positively  this question by showing that if $G$ is closed, then 
$\depth S/J_G^i=\depth S/\ini_<(J_G^i)$ for every $i\geq 1.$

A partial positive answer to this question is the following. Let $G$ be a closed graph with the maximal cliques 
$F_i=[a_i,b_i], 1\leq i\leq r,$ ordered as in Theorem~\ref{th:closedchar}~(iii). Assume that $F_1=[1,2],$ in other words, $G$ has a leaf. We claim that 
\[
\depth\frac{S}{J_G^{k+1}}\leq \depth\frac{S}{J_G^k},
\] for every $k\geq 1.$ In order to prove this inequality, we first observe that $J_G^{k+1}:f_{12}=J_G^k$ for all $k.$ Indeed, 
since $J_G^k\subseteq J_G^{k+1}:f_{12},$ it is enough to show that $\ini_<(J_G^{k+1}:f_{12})=\ini_<(J_G^k).$ Let us assume that 
$\ini_<(J_G^{k+1}:f_{12})\supsetneq\ini_<(J_G^k).$ Then there exists a monomial $w\in \ini_<(J_G^{k+1}:f_{12})\setminus\ini_<(J_G^k).$
Let $h\in J_G^{k+1}:f_{12}$ such that $\ini_<(h)=w.$ We may assume that $h=w+h_1,$ where $\ini_<(h_1)<w.$ Then, 
$f_{12}h=f_{12}(w+h_1)\in J_G^{k+1},$ which implies that 
\[\ini_<(f_{12})w=x_1y_2 w\in \ini_<(J_G^{k+1})=(\ini_<(J_G))^{k+1}=(x_1y_2+\ini_<(J_{G-\{1\}}))^{k+1}=\]
\[=(x_1y_2)(\ini_<(J_G))^k+(\ini_<(J_{G-\{1\}}))^{k+1}.\] Since $x_1$ and $y_2$ do not divide any of the minimal monomial  generators of 
$\ini_<(J_{G-\{1\}})$ and $w\notin (\ini_<(J_G))^k,$ thus $w\notin \ini_<(J_{G-\{1\}}))^{k+1},$ we get 
$\ini_<(f_{12})w\in (x_1y_2)(\ini_<(J_G))^k$ which yields $w\in (\ini_<(J_G))^k,$ a contradiction. Therefore, we have proved the equality $J_G^{k+1}:f_{12}=J_G^k$. We consider the exact sequence of $S$--modules
\begin{equation}\label{eq:Q1}
0\to \frac{S}{J_G^{k+1}:f_{12}}=\frac{S}{J_G^k}\to \frac{S}{J_G^{k+1}}\to \frac{S}{(J_G^{k+1},f_{12})}\to 0.
\end{equation}
Since $J_G^{k+1}=(J_{G-\{1\}}+f_{12})^{k+1}=J_{G-\{1\}}^{k+1}+f_{12}J_G^k,$ it follows that 
$(J_G^{k+1},f_{12})=(J_{G-\{1\}}^{k+1},f_{12}).$ But $f_{12}$ is obviously regular on $S/J_{G-\{1\}}^{k+1},$ thus 
\[
\depth\frac{S}{(J_G^{k+1},f_{12})}=\depth\frac{S}{(J_{G-\{1\}}^{k+1},f_{12})}=\depth\frac{S}{J_{G-\{1\}}^{k+1}}-1.
\] As $G-\{1\}$ is an induced subgraph of $G,$ by \cite[Proposition 3.3]{JKS} it follows that 
$\depth S/J_{G-\{1\}}^{k+1}\geq \depth S/J_G^{k+1}.$ Therefore, we obtained 
\[
\depth\frac{S}{(J_G^{k+1},f_{12})}\geq \depth \frac{S}{J_G^{k+1}}-1. 
\] Depth Lemma applied to sequence (\ref{eq:Q1}) gives the desired inequality.\qed 

A conjecture related to binomial edge ideals on closed graphs which is still open has been given in \cite{EHH}. Namely, it was conjectured that if $G$ is a closed graph, then $J_G$ and $\ini_<(J_G)$ have the same graded Betti numbers. This conjecture is still open. On the other hand, we have noticed in several computer experiments that the same is true for small powers of $J_G.$ Therefore, we are tempted to suggest the following conjecture which extends that one from \cite{EH}.

\begin{Conjecture}\label{Conj}
Let $G$ be a closed graph. Then, for every $i\geq 1,$ $J_G^i$ and $(\ini_<(J_G))^i=\ini_<(J_G^i)$ have the same graded Betti numbers.
\end{Conjecture}

Let us remark in support of our conjecture  that in the previous sections we proved that $\reg J_G^i=\reg(\ini_<(J_G))^i$ for $G$ closed 
and $\depth J_G^i=\depth(\ini_<(J_G))^i$ for $G$ closed and such that  $J_G$ is Cohen-Macaulay. Of course, if the above conjecture is true, then it solves also Question~\ref{Q1}. 

In addition, we note that Conjecture~\ref{Conj} is true for the "extreme" cases, namely when $G$ is a complete graph or a path.

Indeed, if $G=K_n,$ then $(\ini_<(J_G))^i$ has a linear resolution for every $i\geq 1.$ Since $(\ini_<(J_G))^i=\ini_<(J_G^i),$ it follows that $J_G^i$ has a linear resolution for $i\geq 1.$ As the Hilbert functions of $J_G^i$ and $\ini_<(J_G^i)$ coincide, we derive 
that the conjecture is true when $G=K_n$. 
Let us consider $G=P_n$ with the edges $\{i,i+1\}, 1\leq i \leq n-1.$ Then $\ini_<(J_G)$ and $J_G$ are complete intersections  generated in degree $2$ and the conjecture is true by \cite{GT}.

When $G$ is a closed graph such that $J_G$ is Cohen-Macaulay, we have $\beta_{ij}(S/J_G)=\beta_{ij}(S/\ini_<(J_G))$ for all $i,j$ by 
\cite[Proposition 3.2]{EHH}. A possible strategy to prove Conjecture~\ref{Conj} in this case is the following.
We begin with the following nice consequence of  Lemma~\ref{lm:regseq}.

\begin{Corollary}\label{cor:reg}
With the notation from Lemma~\ref{lm:regseq}, we have 
\[
\beta_{ij}^S\left(\frac{S}{J_G^k}\right)=\beta_{ij}^{S(G')}\left(\frac{S(G')}{J_{G'}^k}\right)
\]
and
\[
\beta_{ij}^S\left(\frac{S}{\ini_<(J_G^k)}\right)=\beta_{ij}^{S(G')}\left(\frac{S(G')}{\ini_<(J_{G'}^k)}\right),
\] for all $i,j,$ and $k\geq 1.$
\end{Corollary}
This corollary implies that Conjecture~\ref{Conj} holds for the  closed graphs with Cohen-Macaulay binomial edge ideal once we show that 
for every graph $H$ whose connected components are complete graphs we have 
\[
\beta_{ij}^{S(H)}\left(\frac{S(H)}{J_H^k}\right)=\beta_{ij}^{S(H)}\left(\frac{S(H)}{\ini_<(J_H^k)}\right)
\] for all $i,j,$ and $k\geq 1.$

In Proposition~\ref{pr:nonCMblock}, we proved that if $G$ is a connected block graph, then $J_G^i$ is  Cohen-Macaulay for 
every $i\geq 2$ if and only if $G$ is a path. Then we may ask the following.

\begin{Question}\label{Q2}
Let $G$ be a connected (chordal) graph. Is it true that $J_G^i$ is  Cohen-Macaulay for 
every $i\geq 2$ if and only if $G$ is a path?
\end{Question}

The net  graph $N$ (see Figure~\ref{H1andH2}) which plays an important role in Theorem~\ref{thm:netfree} has the nice property that 
$J_N^{(2)}$ is Cohen-Macaulay. Of course, $J_N^2$ is not Cohen-Macaulay. This naturally yields the following.

\begin{Problem}  
Classify all the block graphs with the property that the second symbolic power of the associated binomial edge ideal is Cohen-Macaulay.
\end{Problem} 

The last question is inspired by Theorem~\ref{thm:netfree}. 

\begin{Question}\label{Q3}
Let $G$ be a graph. Is it true that the following conditions are equiva\-lent?
\begin{itemize}
\item [\emph{(a)}] $J_G^i=J_G^{(i)}$ for all $i\geq 2,$
\item [\emph{(b)}] $J_G^i=J_G^{(i)}$ for some $i\geq 2,$
\item [\emph{(c)}] $J_G^2=J_G^{(2)},$
\item [\emph{(d)}] $G$ is net-free.
\end{itemize}
\end{Question}

We have noticed in computer experiments that for every graph G with at most $8$
vertices the  equivalence (c)$\Leftrightarrow$ (d) holds. Moreover, the implications  
(a)$\Rightarrow$(c)$\Rightarrow$(b)$\Rightarrow$(d) hold and they   can be shown 
similarly as in the proof of Theorem~\ref{thm:netfree}.

\section*{Acknowledgment}

We thank Amin Fakhari and Arvind Kumar for their valuable comments on Section~\ref{S:Open}. We  acknowledge the use of 
\textsc{Macaulay2} \cite{Mac} 
for our computations.

\end{document}